\documentclass{article}
\usepackage{amssymb}
\usepackage{amsmath}
\usepackage{amsfonts}
\usepackage{amscd}
\usepackage{epsfig}
\usepackage{amsmath,amstext,amsthm,amsbsy,amssymb}

\newcommand{\be}{{\bf e}}
\newcommand{\bi}{{\bf i}}

\newcommand{\bc}{{\mathbb C}}
\newcommand{\bn}{{\mathbb N}}
\newcommand{\bz}{{\mathbb Z}}

\newcommand{\br}{{\mathbb R}}
\newcommand{\bh}{{\mathbb H}}

\newtheorem{thm}{Theorem}[section]

\newtheorem{pro}{Proposition}[section]

\newtheorem{defi}{Definition}[section]
\newtheorem{exam}{Example}[section]

\begin{document}

\title{The Euler's formula and root formula of Clifford Algebra $C\ell_2$}
\author{Wensheng Cao, Huihui Cao, Ronglan Zheng}
\date{}
\maketitle

\bigskip
{\bf Abstract.} \,\,  In this paper, we express Euler's formula and De Moivre's formula for Clifford algebra $C\ell_2$ and find nth roots of an element in  Clifford algebra $C\ell_2$.

\vspace{1mm}\baselineskip 12pt

{\bf Mathematics Subject Classification.}\ \ Primary 15A24; Secondary  15A33.

{\bf Keywords.} \ \ Euler's formula,  De Moivre's formula,  Roots of Clifford algebra $C\ell_2$

\section{Introduction}

Clifford introduced a new multiplication rule into Grassmann's exterior algebra, by means of an orthonormal basis, and finally extended them to Clifford algebra $C\ell_{p,q}$.
\begin{defi}(c.f.\cite{Lou}) The Clifford algebra  $C\ell_{p,q}$ with $p+q=n$  is generated by the orthonormal basis $\{i_1,\cdots,i_n\}$ of $\br^{p,q}$ with  the multiplication rules
	\begin{equation}\label{mulrul}i_t^2=1,\,1\le t\le p,\,i_t^2=-1,\,p<t\le n,\,i_ti_m=-i_mi_t,\,t<m.\end{equation}
\end{defi}
 Clifford algebra is a generalization of complex number algebra, quaternion algebra and outer algebra, which integrates the two operations of inner product and outer product. In recent years, Clifford algebra has made brilliant achievements in differential geometry, theoretical physics, classical analysis, etc. It is a core tool of studying modern theoretical mathematics and physics\cite{lib11, Lou,moz06,Moz}.

 Let $\br,\bc$, $\bh$ and $\bh_s$  be  respectively the real numbers, the complex numbers, the quaternions and the split quaternions.  Then we have $\bc\cong C\ell_{0,1}$ , $\bh\cong C\ell_{0,2}$ and $\bh_s\cong C\ell_{1,1}$.

In this paper we focus on the Clifford algebra $C\ell_{2}:=C\ell_{2,0}$.   The basis of $C\ell_{2}$ over $\br$ is
\begin{equation}\label{basise}
	\be_0=1,\,\be_1=i_1,\,\be_2=i_2,\,\be_3=i_1i_2,
\end{equation}
which satisfies
\begin{equation}\label{rule}
\be_1^2=\be_2^2=-\be_3^2=1,\,\, \be_1\be_2=\be_3=-\be_2\be_1,\,\, \be_2\be_3=-\be_1=-\be_3\be_2,\,\, \be_3\be_1=-\be_2=-\be_1\be_3.
\end{equation}

\begin{defi}
For $a=a_0+a_1\be_1+a_2\be_2+a_3\be_3\in C\ell_2$,  where $a_i\in \br,i=0,1,2,3$. We define associated notations and maps of $a$:

the conjugate of $a$: $\overline{a}=a_0-a_1\be_1-a_2\be_2-a_3\be_3$;

the real part of $a$: $\Re(a)=(a+\overline{a})/2=a_0$;

the imaginary part of $a$: $\Im(a)=(a-\overline{a})/2=a_1\be_1+a_2\be_2+a_3\be_3$;

the map $I$: $C\ell_2\longmapsto\br$: $I_a=\overline{a}a=a\overline{a}=a_0^2-a_1^2-a_2^2+a_3^2$;

the norm of $a$: $N_a=\sqrt{\left|I_a\right|}$;

the map $V$: $C\ell_2\longmapsto\br$: $V_a=a_1^2+a_2^2-a_3^2.$
\end{defi}

For $a=a_0+a_1\be_1+a_2\be_2+a_3\be_3,b=b_0+b_1\be_1+b_2\be_2+b_3\be_3\in  C\ell_2$ with $a_i,b_i\in \br$, by the multiplication rule (\ref{rule}), we have
\begin{align*}
	ab&=a_0b_0+a_1b_1+a_2b_2-a_3b_3\\
	&+(a_0b_1+a_1b_0-a_2b_3+a_3b_2)\be_1\\
	&+(a_0b_2+a_1b_3+a_2b_0-a_3b_1)\be_2\\
	&+(a_0b_3+a_1b_2-a_2b_1+a_3b_0)\be_3.
\end{align*}

The real exponential function $e^x$   has a power series expansion $$e^x=\sum_{k=0}^{\infty}\frac{x^k}{k!}=1+x+\cdots+\frac{x^k}{k!}+\cdots,\,\forall x\in\br.$$
Extending this series to the complex domain $\bc=\{z:z=x+y\bi,x,y\in \br\}$, one get the complex exponential function
\begin{equation}\label{cexp}e^{z}=e^xe^{\bi y}=e^x(\cos y+\bi\sin y), \forall z=x+y\bi\in\bc\end{equation}
and  the Euler's formula
\begin{equation}\label{ceuler}e^{\bi\theta}=(1-\frac{\theta^2}{2!}+\frac{\theta^4}{4!}+\cdots)+\bi(\theta-\frac{\theta^3}{3!}+\frac{\theta^5}{5!}+\cdots)=\cos\theta+\bi\sin\theta.\end{equation}
The De Moivre's formula in $\bc$ reads
\begin{equation}\label{cdem}(\cos\theta+\bi\sin\theta)^n=\cos n\theta+\bi\sin n\theta.\end{equation}
The nth roots of $z=re^{\bi\theta}\in \bc$ are
\begin{equation}\label{csqrt}\sqrt[n]{z}=\sqrt[n]{r}\big( \cos \frac{\theta+2k\pi}{n}+\bi\sin\frac{\theta+2k\pi}{n}\big),k=0,\cdots,n-1.\end{equation}

The above formulas in complex number field have been extended to other algebra systems.
The roots of a quaternion were given by Niven \cite{Niv} and Brand \cite{Bra}. Brand proved De Moivre's theorem and used it to
find nth roots of a quaternion.  Cho \cite{Cho} obtained Euler and De Moivre's formulas for quaternions.  \"{O}zdemir \cite{Moz} studied the nth roots and powers of a split quaternion under certain conditions. Cao and Chang \cite{cao} studied the root of a split quaternion in $Z(\bh_s)$.

In this paper, we will extended the above formulas  to  Clifford algebra $C\ell_2$.  In Section 2, we express Euler and De Moivre's formulas for Clifford algebra $C\ell_2$ and obtain some properties of  the exponential function in $C\ell_2$.
 In Section 3, we will get the nth roots in $C\ell_2$.

\section{Euler's formula  for the Clifford algebra $C\ell_2$}

In this section we will define exponential function in $C\ell_2$.  We will obtain the Euler's formula  and De Moivre's formulas for  Clifford algebra $C\ell_2$.

We define
\begin{equation}\label{vc}V=\{v=v_1 \be_1+v_2 \be_2+v_3 \be_3:v_1,v_2,v_3\in \br\}.\end{equation}
The set $V$  is the imaginary part of Clifford algebra $C\ell_2$, which  can be  identified with the Minkowski space $\br^{2,1}$ with the following Lorentzian inner product:
\begin{equation}\label{inn}\left\langle u,v\right\rangle_L=u_1v_1+u_2v_2-u_3v_3, \forall u,v\in V.\end{equation}
We define the following subsets of $V$:
\begin{equation}\label{E1im}E_1=\{\varepsilon=\epsilon_1 \be_1+\epsilon_2 \be_2+\epsilon_3 \be_3\in V:\epsilon_1^2+\epsilon_2^2-\epsilon_3^2=1\};\end{equation}
\begin{equation}\label{E0im}E_0=\{\varepsilon=\epsilon_1 \be_1+\epsilon_2 \be_2+\epsilon_3 \be_3\in V:\epsilon_1^2+\epsilon_2^2-\epsilon_3^2=0\};\end{equation}
\begin{equation}\label{E-1im}E_{-1}=\{\varepsilon=\epsilon_1 \be_1+\epsilon_2 \be_2+\epsilon_3 \be_3\in V:\epsilon_1^2+\epsilon_2^2-\epsilon_3^2=-1\}.\end{equation}
Note that $$\left\langle u,v\right\rangle_L=\Re\big(uv\big),\forall u,v\in V$$
and
\begin{equation}\label{imim}V_u=u^2=\Re\big(u^2\big)=-I_{u},\forall u\in V .\end{equation}

\begin{defi}
	We define the exponential function in  $C\ell_2$ as following: $$e^a=\sum_{k=0}^{\infty}\frac{a^k}{k!}=1+a+\cdots+\frac{a^k}{k!}+\cdots,\,\forall a\in C\ell_2.$$
\end{defi}

\begin{thm}(Euler's formulas)\label{Euler} Let $\theta\in \br$.
\begin{itemize}
	\item[(i)] 	If $\varepsilon\in E_1$  then  $$e^{\varepsilon\theta}=(1+\frac{\theta^2}{2!}+\frac{\theta^4}{4!}+\cdots)+\varepsilon(\theta+\frac{\theta^3}{3!}+\frac{\theta^5}{5!}+\cdots)=\cosh\theta+\varepsilon\sinh\theta$$

\item[(ii)]	If $\varepsilon\in E_{-1}$  then $$e^{\varepsilon\theta}=(1-\frac{\theta^2}{2!}+\frac{\theta^4}{4!}-\cdots)+\varepsilon(\theta-\frac{\theta^3}{3!}+\frac{\theta^5}{5!}-\cdots)=\cos\theta+\varepsilon\sin\theta$$

\item[(iii)] If $\varepsilon\in E_0$ then  $$e^{\varepsilon\theta}=1+\varepsilon\theta.$$
	\end{itemize}
\end{thm}
\begin{proof}
	If $\varepsilon\in E_{1}$ then   $$\varepsilon^{2k}=1,\ \ \varepsilon^{2k+1}=\varepsilon.$$ Therefore
	$$e^{\varepsilon\theta}=\sum_{k=0}^{\infty}\frac{(\varepsilon\theta)^k}{k!}=(1+\frac{\theta^2}{2!}+\frac{\theta^4}{4!}+\cdots)+\varepsilon(\theta+\frac{\theta^3}{3!}+\frac{\theta^5}{5!}+\cdots)=\cosh\theta+\varepsilon\sinh\theta.$$ This proves (i).

		If $\varepsilon\in E_{-1}$  then $$\varepsilon^{2k}=(-1)^k,\ \ \varepsilon^{2k+1}=(-1)^k\varepsilon.$$
	 Therefore
	$$e^{\varepsilon\theta}=\sum_{k=0}^{\infty}\frac{(\varepsilon\theta)^k}{k!}=(1-\frac{\theta^2}{2!}+\frac{\theta^4}{4!}-\cdots)+\varepsilon(\theta-\frac{\theta^3}{3!}+\frac{\theta^5}{5!}-\cdots)=\cos\theta+\varepsilon\sin\theta$$ This proves (ii).	
		
		 If $\varepsilon\in E_0$ then  $$\varepsilon^{k+2}=0.$$
	Therefore
	$$e^{\varepsilon\theta}=\sum_{k=0}^{\infty}\frac{(\varepsilon\theta)^k}{k!}=1+\varepsilon\theta.$$ This proves (iii).	
\end{proof}

If $I_a<0$ then $V_a>0$ since $I_a=a_0^2-V_a$. So we can divide  $C\ell_2$ into the following five subsets:
$$S_1=\{a\in C\ell_2: I_a>0\mbox{ and } V_a>0\},$$
$$S_2=\{a\in C\ell_2: I_a>0\mbox{ and } V_a<0\},$$
$$S_3=\{a\in C\ell_2: I_a<0\mbox{ and } V_a>0\},$$
$$S_4=\{a\in C\ell_2:I_a=0\},$$
$$S_5=\{a\in C\ell_2:V_a=0\}.$$
Note that \begin{equation}C\ell_2=\cup_{i=1}^{5} S_i,\,\, (\cup_{i=1}^{3} S_i)\cap (\cup_{i=4}^{5} S_i)=\emptyset,\,\, S_4\cap S_5=E_0.\end{equation}
If $a\in C\ell_2-S_4$ then  $I_a\neq 0$ and  \begin{equation}a^{-1}=\frac{\overline{a}}{I_a}
\end{equation}
 is the inverse of $a$.  By  the above formula,  we have the following proposition.

\begin{pro}
	\label{inversetype}
	\begin{itemize}
		\item[(i)]   If $a=N_a(\cosh\theta+\varepsilon\sinh\theta)\in S_1$ then
		$$a^{-1}=N_a^{-1}(\cosh(-\theta)+\varepsilon\sinh(-\theta))\in S_1.$$

		\item[(ii)] If $a=N_a(\cos\theta+\varepsilon\sin\theta)\in S_2$ then
			$$a^{-1}=N_a^{-1}(\cos(-\theta)+\varepsilon\sin(-\theta))\in S_2.$$

		\item[(iii)]  If $a=N_a(\sinh\theta+\varepsilon\cosh\theta)\in S_3$
		then
			$$a^{-1}=N_a^{-1}(\sinh(-\theta)+\varepsilon\cosh(-\theta))\in S_3.$$

		\item[(iv)]  If $a=a_0+a_1\be_1+a_2\be_2+a_3\be_3\in S_5-S_4$ then $a^{-1}=a_0^{-1}-a_0^{-2}\Im(a).$
	\end{itemize}
	\end{pro}

Let $\bz$ be the set of integers and $\bn$ be the set of natural numbers.
\begin{thm} (De Moivre's formulas)\label{dmoiv}
	\begin{itemize}
		\item[(i)]   $a=a_0+a_1\be_1+a_2\be_2+a_3\be_3\in S_1$ can be written in the form
		\begin{equation}\label{ch1}
			a=a_0+\Im(a)=N_a\big(\frac{a_0}{N_a}+\frac{\Im(a)}{\sqrt{V_a}}\frac{\sqrt{V_a}}{N_a}\big)=N_a(\cosh\theta+\varepsilon\sinh\theta),
		\end{equation}
		where \begin{equation}\label{ch2}\cosh\theta=\frac{a_0}{N_a},\, \sinh\theta=\frac{\sqrt{V_a}}{N_a},\, \varepsilon=\frac{\Im(a)}{\sqrt{V_a}}\in E_1\end{equation}
		and
		\begin{equation}\label{1}
		a^n=N_a^n(\cosh n\theta+\varepsilon\sinh n\theta), \forall\, n\in \bz.
	\end{equation}

		\item[(ii)] $a=a_0+a_1\be_1+a_2\be_2+a_3\be_3\in S_2$ can be written in the form
		\begin{equation}\label{cos1}
			a=a_0+\Im(a)=N_a\big(\frac{a_0}{N_a}+\frac{\Im(a)}{\sqrt{-V_a}}\frac{\sqrt{-V_a}}{N_a}\big)=N_a(\cos\theta+\varepsilon\sin\theta),
		\end{equation}
		where \begin{equation}\label{cos2}\cos\theta=\frac{a_0}{N_a},\, \sin\theta=\frac{\sqrt{-V_a}}{N_a},\, \varepsilon=\frac{\Im(a)}{\sqrt{-V_a}}\in E_{-1}\end{equation}
			and
		\begin{equation}\label{r1}
			a^n=N_a^n(\cos n\theta+\varepsilon\sin n\theta),\forall\, n\in \bz.
		\end{equation}

		\item[(iii)]  $a=a_0+a_1\be_1+a_2\be_2+a_3\be_3\in S_3$ can be written in the form
	\begin{equation}\label{ch3}
		a=a_0+\Im(a)=N_a\big(\frac{a_0}{N_a}+\frac{\Im(a)}{\sqrt{V_a}}\frac{\sqrt{V_a}}{N_a}\big)=N_a(\sinh\theta+\varepsilon\cosh\theta),
	\end{equation}
	where \begin{equation}\label{ch4}\sinh\theta=\frac{a_0}{N_a},\, \cosh\theta=\frac{\sqrt{V_a}}{N_a},\, \varepsilon=\frac{\Im(a)}{\sqrt{V_a}}\in E_1,\end{equation}
	and

	$$a^n=\begin{cases} (N_a)^{2k+1}\big(\sinh{(2k+1)\theta}+\varepsilon\cosh{(2k+1)\theta}\big),& n=2k+1,\forall\, k\in \bz; \\(N_a)^{2k}\big(\cosh{2k\theta}+\varepsilon\sinh{2k\theta}\big), & n=2k,\forall\, k\in \bz. \end{cases}$$

\item[(iv)]  Let $a=a_0+a_1\be_1+a_2\be_2+a_3\be_3\in S_4$. Then \begin{equation}a^n=(2\Re(a))^{n-1}a, \forall\, n\in \bn. \end{equation}

\item[(v)]  If $a=a_0+a_1\be_1+a_2\be_2+a_3\be_3\in S_5\cap S_4=E_0$ then \begin{equation}a^n=\Im(a)^n=0, \forall\, n\geq 2; \end{equation}
and if $a=a_0+a_1\be_1+a_2\be_2+a_3\be_3\in S_5-S_4$ then
\begin{equation}a^n=(\Re(a))^{n}+n(\Re(a))^{n-1}\Im(a),\, \forall\, n\in \bz. \end{equation}
	\end{itemize}
\end{thm}
\begin{proof}

For $a\in S_1$, we can write $a$ as $$a=N_a(\cosh\theta+\varepsilon\sinh\theta),$$ where $\theta$ and $\varepsilon$ are given by (\ref{ch2}).
Assume that $a^n=N_a^n(\cosh n\theta+\varepsilon\sinh n\theta)$ holds for $n\in \bn$. Then
\begin{align*}a^{n+1}&=N_a^n(\cosh n\theta+\varepsilon\sinh n\theta)N_a(\cosh\theta+\varepsilon\sinh\theta)\\
&=N_a^{n+1}\big(\cosh n\theta\cosh \theta+\varepsilon\sinh n\theta\sinh \theta\big)\\
&=N_a^{n+1}(\cosh (n+1)\theta+\varepsilon\sinh (n+1)\theta).
\end{align*}
We conclude the proof of (i) by  induction for positive integers.  Note that $$a^{-1}=N_a^{-1}\big(\cosh(-\theta)+\varepsilon\sinh (-\theta)\big).$$
So we have
$$a^{-n}=N_a^{-n}\big(\cosh(-n\theta)+\varepsilon\sinh (-n\theta)\big).$$
This proves (i). Similarly, we can prove (ii).

For $a\in S_3$, $a$ can be written in the form
$$a=N_a(\sinh\theta+\varepsilon\cosh\theta),$$
where $\theta$ and $\varepsilon$ are given by (\ref{ch4}). Note that
$a^2=N_a^2(\cosh 2\theta+\varepsilon\sinh 2\theta)$. By the properties of functions $\cosh x$ and  $\sinh x$, we can prove (iii).

Noting  that $a=2\Re(a)-\overline{a}$,   we have $$a^2=a(2\Re(a)-\overline{a})=2\Re(a)a-I_a.$$  If $I_a=0$ then $a^2=2\Re(a)a$. By induction, we have  $a^n=(2\Re(a))^{n-1}a$. This proves (iv).

Note that $a=a_0+\Im(a)$ and $a_0 \Im(a)=\Im(a) a_0$. If $V_a=0$ then $$\Im(a)^k=0,\forall k\geq2.$$
Therefore for $n\geq 0$ we have
$$a^n=(a_0+\Im(a))^n=\sum_{k=0}^nC_n^k a_0^{n-k}\Im(a)^k=a_0^{n}+n a_0^{n-1}\Im(a).$$
Furthermore if $a=a_0+a_1\be_1+a_2\be_2+a_3\be_3\in S_5-S_4$ then  $$a^{-1}=\frac{\overline{a}}{I_a}=a_0^{-1}-a_0^{-2}\Im(a)\in S_5$$ and

$$a^{-n}=a_0^{-n}+na_0^{-(n-1)}\big(-a_0^{-2}\Im(a)\big)=a_0^{-n}+(-n)a_0^{-n-1}\Im(a).$$
This proves (v).
\end{proof}

Note that
$$e^a=\sum_{k=0}^{\infty}\frac{(a_0+\Im(a))^k}{k!}=\sum_{m=0}^{\infty}\frac{a_0^m}{m!}\sum_{n=0}^{\infty}\frac{\Im(a)^n}{n!}=e^{a_0}e^{\Im(a)}.$$  By Theorem \ref{Euler}, we have the following theorem.

\begin{thm}(Exponential function)\label{exp}
		 	$$e^a=\begin{cases} e^{a_0}(\cosh{\sqrt{V_a}}+\frac{\Im(a)}{\sqrt{V_a}}\sinh{\sqrt{V_a}}),& V_a>0; \\e^{a_0}(\cos{\sqrt{-V_a}}+\frac{\Im(a)}{\sqrt{-V_a}}\sin{\sqrt{-V_a}}),& V_a<0;\\
	 	e^{a_0}(1+\Im(a)),& V_a=0.
 	 \end{cases}$$
\end{thm}

\section{The root of an element in  $C\ell_2$} \label{rootse}
In this section we want to find roots of a Clifford algebra $C\ell_2$ by applying  De Moivre's formulas.

 We can verify the following proposition.
\begin{pro}\label{pro1.1}For $n\in \bn$ and $t\in \br$, we have the following properties.
\begin{itemize}
	\item[(i)] 	$$I_{ab}=I_aI_b,\,I_{-a}=I_{\bar{a}}, I_{ta}=t^2I_a, \, I_{a^n}=I_a^n;$$
	\item[(ii)] 	$$V_{-a}=V_{\bar{a}},\,\, V_{ta}=t^2V_a.$$
	\end{itemize}
\end{pro}

\begin{pro}
\label{Vcase}For $n\in \bn$, we have the following properties.
	\begin{itemize}
		\item[(i)]   If $a=N_a(\cosh\theta+\varepsilon\sinh\theta)\in S_1$ then
	$V_{a^n}=I_a^n(\sinh n\theta)^2.$

		\item[(ii)] If $a=N_a(\cos\theta+\varepsilon\sin\theta)\in S_2$ then
		$V_{a^n}=-I_a^n(\sin n\theta)^2.$

		\item[(iii)]  If $a=N_a(\sinh\theta+\varepsilon\cosh\theta)\in S_3$
	then
		
		$$V(a^n)=\begin{cases} (I_a)^{2k+1}(\cosh{(2k+1)\theta})^2,& n=2k+1 \\(I_a)^{2k}(\sinh{2k\theta})^2, & n=2k. \end{cases}$$

		\item[(iv)]  If $a=a_0+a_1\be_1+a_2\be_2+a_3\be_3\in S_4$ then $V_{a^n}=(2a_0)^{2n-2}V_a.$
		
		\item[(v)]  If $a=a_0+a_1\be_1+a_2\be_2+a_3\be_3\in S_5$ then $V_{a^n}=0.$
		\end{itemize}

\end{pro}

By Theorem \ref{dmoiv} and Proposition \ref{Vcase}, we have the following proposition.
\begin{pro}\label{type}For $n,k\in \bn$, we have the following properties.
	\begin{itemize}
		\item [(i)]
 If $a\in S_i$ then  $a^n\in S_i$ for $i=1,4,5$;
  \item [(ii)]
  If $a=N_a(\cos\theta+\varepsilon\sin\theta)\in S_2$ then  $a^n\in S_2$ provided $\sin n\theta\neq 0$, and  $a^n\in S_5$ provided $\sin n\theta=0$;	
  	\item [(iii)]
 If $a\in S_3$ then $a^{2k}\in S_1$ and  $a^{2k+1}\in S_3$ for $k\in \bz$.

  \end{itemize}
\end{pro}

We now are ready to find the root of the equation  $w^n=a$ with $n\geq 2$.

\begin{thm}\label{prot} Let $a=N_a(\cosh\theta+\varepsilon\sinh\theta)\in S_1$.

\begin{itemize}
	\item [(i)]	
	If $n$ is  odd  then the equation $w^n=a$ has only one root $$w=\sqrt[n]{N_a}(\cosh\frac{\theta}{n}+\varepsilon\sinh\frac{\theta}{n});$$
	\item [(ii)] If $n$ is  even  then the equation $w^n=a$ has  two  roots $$w=\sqrt[n]{N_a}(\cosh\frac{\theta}{n}+\varepsilon\sinh\frac{\theta}{n})$$
	and $$w=\sqrt[n]{N_a}(\sinh\frac{\theta}{n}+\varepsilon\cosh\frac{\theta}{n}).$$
\end{itemize}

\end{thm}
\begin{proof} If $n$ is odd, by Theorem \ref{dmoiv} and Proposition \ref{type},  we assume that $w=N(\cosh \varphi+\varepsilon\sinh \varphi)$ and therefore $$w^n=N^n(\cosh{n\varphi}+\varepsilon\sinh{n\varphi}).$$
Thus $N^n=N_a$ and $\varphi=\frac{\theta}{n}$. This proves (i).

   If $n$ is even,   we assume that $w=N(\cosh \varphi+\varepsilon\sinh \varphi)\in S_1$ or $w=N(\sinh \varphi+\varepsilon\cosh \varphi)\in S_3$.
   If $w\in S_1$ then $w=\sqrt[n]{N_a}(\cosh\frac{\theta}{n}+\varepsilon\sinh\frac{\theta}{n})$;
      If $w\in S_3$ then $w=\sqrt[n]{N_a}(\sinh\frac{\theta}{n}+\varepsilon\cosh\frac{\theta}{n})$.
      This proves (ii).
\end{proof}

\begin{exam}  Let $a=\sqrt{2}+7\be_1+4\be_2+8\be_3\in S_1$. Then the solution of $w^3=a$ is
$$w=\cosh\frac{\ln(1+\sqrt{2})}{3}+\varepsilon\sinh\frac{\ln(1+\sqrt{2})}{3}$$
and the two solutions of $w^2=a$ are
$$w_1=\cosh\frac{\ln(1+\sqrt{2})}{3}+\varepsilon\sinh\frac{\ln(1+\sqrt{2})}{3}$$
and
$$w_2=\sinh\frac{\ln(1+\sqrt{2})}{3}+\varepsilon\cosh\frac{\ln(1+\sqrt{2})}{3},$$
where $\varepsilon=7\be_1+4\be_2+8\be_3$.

\end{exam}

\begin{thm}\label{prot} Let $a=N_a(\cos\theta+\varepsilon\sin\theta)\in S_2$. Then the equation $w^n=a$ has $n$ roots:
$$w=\sqrt[n]{N_a}(\cos\frac{\theta+2m\pi}{n}+\varepsilon\sin\frac{\theta+2m\pi}{n})$$
for $m=0,1,2,\cdots,n-1$.
\end{thm}
\begin{proof}  By  Proposition \ref{type},  we assume that $w=N(\cos\varphi+\varepsilon\sin\varphi)$. Then $$w^n=N^n(\cos n\varphi+\varepsilon\sin n\varphi).$$
	Therefore
$$N_a=N^n, \, \cos{n\varphi}=\cos\theta \ \ and \ \  \sin{n\varphi}=\sin\theta$$
By using of Theorem \ref{dmoiv}, the $n$th roots of $a$ are
$$w_m=\sqrt[n]{N_a}(\cos\frac{\theta+2m\pi}{n}+\varepsilon\sin\frac{\theta+2m\pi}{n})$$
for $m=0,1,2,\cdots,n-1$.
\end{proof}

\begin{exam} Let $a=1-\be_3\in S_2$. Then the solutions of $w^4=a$ are
$$w_m=\cos\frac{\frac{\pi}{4}+2m\pi}{4}+\varepsilon\sin\frac{\frac{\pi}{4}+2m\pi}{4},\,\,\,\,\,m=0,1,2,3,$$
where $\varepsilon=-\be_3$.

\end{exam}

\begin{thm}\label{prot} Let $a=N_a(\sinh\theta+\varepsilon\cosh\theta)\in S_3$. Then the equation $w^n=a$ has
\begin{itemize}
	\item [(i)]	
	 no root  if $n$ is an even number;
	
	\item [(ii)]	 only one  root $w=\sqrt[n]{N_a}(\sinh\frac{\theta}{n}+\varepsilon\cosh\frac{\theta}{n})$  if $n$ is an odd number.
	\end{itemize}
\end{thm}

\begin{proof} By Proposition \ref{type}, if $n$ is an even number then $w^n=a$  has no root.  For the case of $n$ being odd, let
 $w=N(\sinh\varphi+\varepsilon\cosh\varphi)$ be a root of the equation $w^n=a$. Then
	$$w^n=N^n(\sinh{n\varphi}+\varepsilon\cosh{n\varphi})=N_a(\sinh\theta+\varepsilon\cosh\theta)$$
	and we have $\varphi=\frac{\theta}{n}$ and $N=\sqrt[n]{N_a}$.
\end{proof}
\begin{exam}
Let $a=1+\be_1-\be_2\in S_3$. Then the solution of $w^3=a$ is
$$w=\sinh\frac{\ln(1+\sqrt{2})}{3}+\varepsilon\cosh\frac{\ln(1+\sqrt{2})}{3},$$ where $\varepsilon=\frac{\be_1-\be_2}{\sqrt{2}}$.
\end{exam}

\begin{thm}\label{thms4}Let $a=a_0+\Im(a)\in S_4$. Then the equation $w^n=a$ has solutions:
	\begin{itemize}
		\item [(i)]	
		$w=\sqrt[n]{\frac{a_0}{2^{n-1}}}+\frac{\Im(a)}{2^{n-1}\left(\sqrt[n]{\frac{a_0}{2^{n-1}}}\right)^{n-1}}$ if $a_0\neq 0$ and  $n$\mbox{ is  odd};
		
		\item [(ii)]	$w=\sqrt[n]{\frac{a_0}{2^{n-1}}}+\frac{\Im(a)}{2^{n-1}\left(\sqrt[n]{\frac{a_0}{2^{n-1}}}\right)^{n-1}}$ if $a_0>0$ and  $n$\mbox{ is  even};
		
		\item [(iii)]	$w\in E_0$ if $a=0$ and $n\geq 2$.
			\end{itemize}
	\end{thm}

\begin{proof}
By Proposition \ref{type},  we assume that $w=w_0+\Im(w)\in S_4$ is a root of  $w^n=a$.
By Theorem \ref{dmoiv} we have
$$w^n=(2w_0)^{n-1}(w_0+\Im(w))=a_0+\Im(a).$$Therefore
\begin{equation}\label{s4eq}2^{n-1}w_0^n=a_0,\, \, (2w_0)^{n-1}\Im(w)=\Im(a).\end{equation}
If $n$ is even, we need $a_0\ge 0$. If $a_0>0$ then we have $$w_0=\sqrt[n]{\frac{a_0}{2^{n-1}}},\,\, \Im(w)=\frac{\Im(a)}{2^{n-1}(\sqrt[n]{\frac{a_0}{2^{n-1}}})^{n-1}}.$$
This proves (i).

If $a_0=0$ then $a\in S_4\cap S_5=E_0$ and $w_0=0$. By (\ref{s4eq}) and we need $\Im(a)=0.$  Thus $\Im(w)\in S_4$.  So each $w\in E_0$ is a root of $w^n=0$. This proves (iii).

If $n$ is odd with $a_0\neq 0$ then   $$w_0=\sqrt[n]{\frac{a_0}{2^{n-1}}},\,\, \Im(w)=\frac{\Im(a)}{2^{n-1}(\sqrt[n]{\frac{a_0}{2^{n-1}}})^{n-1}}.$$
This proves (ii).
\end{proof}

\begin{exam} Let $a=2+5\be_1+10\be_2+11\be_3\in S_4$. Then	\begin{itemize}
		\item [(i)]	the solution of $w^3=a$ is
$$w=\sqrt[3]{\frac{1}{2}}+\frac{1}{2}\sqrt[3]{\frac{1}{2}}(5\be_1+10\be_2+11\be_3);$$
\item [(ii)]	the solution of $w^2=a$ is
$$w=\frac{2+5\be_1+10\be_2+11\be_3}{2}.$$
\end{itemize}
\end{exam}

\begin{thm}\label{thms5} Let $a=a_0+a_1\be_1+a_2\be_2+a_3\be_3\in S_5-\{0\}$.  Then the equation $w^n=a$ has solutions:
	\begin{itemize}
		\item [(i)]	
		$w=\sqrt[n]{a_0}+\frac{\Im(a)}{(\sqrt[n]{a_0})^{n-1}}$ if $a_0\neq 0$ and  $n$\mbox{ is  odd};
		
		\item [(ii)]	$w=\sqrt[n]{a_0}+\frac{\Im(a)}{(\sqrt[n]{a_0})^{n-1}}$ if $a_0>0$ and  $n$\mbox{ is  even};
		
			\item [(iii)]	$w=\sqrt[n]{a_0}(\cos\frac{2m\pi}{n}+\varepsilon\sin\frac{2m\pi}{n}),m=0,1,\cdots,n-1, \forall \varepsilon\in E_{-1}$ if  $a=a_0>0$;
		
	\item [(iv)]	$w=\sqrt[n]{-a_0}(\cos\frac{\pi+2m\pi}{n}+\varepsilon\sin\frac{\pi+2m\pi}{n}),m=0,1,\cdots,n-1, \forall \varepsilon\in E_{-1}$ if  $a=a_0<0$.
			\end{itemize}
\end{thm}
\begin{proof}
	By Proposition \ref{type}, the root $w$ may belong to $S_5$ and $S_2$.
	
	  We first assume that $w=w_0+\Im(w)\in S_5$.
	By Theorem \ref{dmoiv} we have
	$$w^n=w_0^{n}+nw_0^{n-1}\Im(w)=a_0+\Im(a).$$
Therefore
	\begin{equation}\label{s5eq}w_0^n=a_0,\, \, nw_0^{n-1}\Im(w)=\Im(a).\end{equation}
	If $n$ is even, we need $a_0\ge 0$. If $a_0>0$ then we have $$w_0=\sqrt[n]{a_0},\,\, \Im(w)=\frac{\Im(a)}{(\sqrt[n]{a_0})^{n-1}}.$$
	This proves (ii).

	If $a_0=0$ then $a\in S_4\cap S_5=E_0$ and $w_0=0$. By (\ref{s5eq}) and we need $\Im(a)=0.$ Thus $a=0$ which has been considered in Theorem \ref{thms4}.

		If $n$ is odd with $a_0\neq 0$ then   $$w_0=\sqrt[n]{a_0},\,\, \Im(w)=\frac{\Im(a)}{(\sqrt[n]{a_0})^{n-1}}.$$
	This proves (i).
		
	 We next assume that $w=N(\cos\varphi+\varepsilon\sin\varphi)\in S_2$.  Then $w^n=N^n(\cos n\varphi+\varepsilon\sin n\varphi)$ and therefore we have
	 $$N^n\cos n \varphi=a_0,\,\,  N^n\sin n \varphi\,\varepsilon  =\Im(a).$$	
	Since $\varepsilon^2=-1$ and $V_a=0$, we need  $\Im(a)=0$ and 	$\sin n \varphi=0$. So we have  $\cos  n \varphi=\pm 1$.

If  $\cos  n \varphi= 1$ then we need $a_0>0$. In this case
  $$N=\sqrt[n]{a_0},\,\, \varphi=\frac{2m\pi}{n}, m=0,1,\cdots,n-1.$$

  So the roots can be written as
  $$w_m=\sqrt[n]{a_0}(\cos\frac{2m\pi}{n}+\varepsilon\sin\frac{2m\pi}{n}),m=0,1,\cdots,n-1, \forall \varepsilon\in E_{-1}.$$
	This proves (iii).

  If  $\cos  n \varphi= -1$ then we need $a_0<0$. In this case
  $$N=\sqrt[n]{-a_0},\,\, \varphi=\frac{\pi+2m\pi}{n}, m=0,1,\cdots,n-1.$$
  So the roots can be written as
  $$w_m=\sqrt[n]{-a_0}(\cos\frac{\pi+2m\pi}{n}+\varepsilon\sin\frac{\pi+2m\pi}{n}),m=0,1,\cdots,n-1, \forall \varepsilon\in E_{-1}.$$
  	This proves (iv).

  	\end{proof}

\begin{exam} (i) Let $a=8+3\be_1-4\be_2+5\be_3\in S_5$. Then the solution of $w^3=a$ is
$$w=\frac{8+3\be_1-4\be_2+5\be_3}{4}.$$
(ii) Let $a=16+3\be_1-4\be_2+5\be_3\in S_5$. Then the solution of $w^4=a$ is
$$w=\frac{16+3\be_1-4\be_2+5\be_3}{8}.$$
(iii) Let $a=1\in S_5$. Then the solutions of $w^4=a$ are
$$w_m=\cos\frac{m\pi}{2}+\varepsilon\sin\frac{m\pi}{2},\,\,m=0,1,2,3,\forall\varepsilon\in E_{-1}.$$
(iv) Let $a=-2$. Then the solutions of $w^3=a$ are
$$w_m=\sqrt[3]{2}(\cos\frac{\pi+2m\pi}{3}+\varepsilon\sin\frac{\pi+2m\pi}{3}),\,\,m=0,1,2,\forall\varepsilon\in E_{-1}.$$

\end{exam}

\section*{Acknowledgements}

This work is supported by Natural Science Foundation of China (11871379), the Innovation Project of Department of Education of Guangdong Province (2018KTSCX231) and Key project of  Natural Science Foundation  of Guangdong Province Universities (2019KZDXM025).

\end{document}